\documentclass[11pt]{article}

\usepackage{array}
\usepackage{booktabs}
\usepackage{longtable}

\newcolumntype{L}[1]{>{\raggedright\arraybackslash}p{#1}}
\newcolumntype{C}[1]{>{\centering\arraybackslash}p{#1}}

\usepackage[a4paper,margin=2.8cm]{geometry}

\title{More (shifted) runners, less loneliness}
\author{%
  Daria Poliakova\\
  \small University of Hamburg\\
  \small\href{mailto:polydarya@gmail.com}{\texttt{polydarya@gmail.com}}
}
\date{}
\usepackage{booktabs}
\usepackage{amsmath,amssymb,amsthm}
\usepackage{mathtools}
\usepackage{graphicx}
\usepackage{enumitem}
\usepackage[hidelinks]{hyperref}

\usepackage{tikz}
\usetikzlibrary{arrows.meta,calc,positioning}

\setlist[itemize]{leftmargin=2em}
\setlist[enumerate]{leftmargin=2.2em}

\usepackage{accents}
\newlength{\dhatheight}

\newtheorem{theorem}{Theorem}
\newtheorem{proposition}[theorem]{Proposition}

\theoremstyle{definition}
\newtheorem{definition}[theorem]{Definition}

\newcommand{\R}{\mathbb R}
\newcommand{\Z}{\mathbb Z}
\newcommand{\norm}[1]{\lVert #1\rVert}

\begin{document}

\maketitle

\begin{abstract}
The shifted lonely runner conjecture was recently disproved by Blanco, Criado and Santos. We give quantitative bounds on its failure as the number of runners grows. The loneliness of a configuration is the maximum, over time, of the distance from the origin to the nearest runner. We write \(1/(n+1+E_n)\) for the infimum of loneliness over configurations of \(n\) runners on the unit circle with distinct positive integer velocities and arbitrary initial shifts.  We prove $E_n\ge\lfloor n/287\rfloor$, which together with the elementary bound \(E_n\leq n-1\) implies that \(E_n\) grows linearly in \(n\). We also show that \(E_n\geq1\) for every \(n\geq95\).
\end{abstract}

\section*{Introduction}

The ordinary lonely runner conjecture (LR), originating in work of Wills \cite{Wills1968} and Cusick \cite{Cusick1973}, states that when \(n+1\) runners start at the same place and time and move around a circle of length \(1\) with distinct constant velocities, each runner will at some time be lonely, i.e., at a distance at least \(1/(n+1)\) from every other runner. For the shifted version (SLR), we fix one runner at the origin and consider \(n\) other runners with distinct positive integer velocities. SLR then asserted that the stationary runner must still be lonely at some time when the others are allowed arbitrary starting positions \cite[Conjecture~10]{BeckHostenSchymura2019}. This change of setup makes a big difference: while LR remains open in general and has been shown to be true for up to fourteen runners \cite{BarajasSerra2007,MalikiosisSantosSchymura2024,Rosenfeld2025Eight,Trakulthongchai2025,Rosenfeld2025Nine,SungkawichaiTrakulthongchai2026,Allikvere2026}, SLR has been disproved by Blanco, Criado and Santos \cite{BlancoCriadoSantos2026}, with smallest counterexample having only five moving runners. Their counterexamples were found using the geometric reformulation in terms of covering radii of lattice zonotopes associated with the velocity vectors \cite{HenzeMalikiosis2017,BlancoCriadoSantos2026}.

Failure of SLR opens the door for new questions. The loneliness of a configuration is the greatest distance from the nearest moving runner that the stationary runner attains over time. Beck and Everett \cite[Section~4]{BeckEverett2026} ask for the infimum of this quantity over all velocity vectors and shifts in small dimensions. Zhang \cite[Question~9.4]{Zhang2026} asks whether SLR fails for every \(n\geq5\), and how the shifted minimizers are related to tight configurations for ordinary LR. Here, we ask the asymptotic question: how badly can SLR fail as the number of runners grows?

We write the infimum considered by Beck and Everett as \(1/(n+1+E_n)\), so that \(E_n\) measures the excess in the reciprocal of the worst possible loneliness. This note provides a partial answer by proving
\[
E_n\ge\lfloor n/287\rfloor.
\]
Together with the elementary upper bound \(E_n\leq n-1\), this shows that \(E_n\) grows linearly in \(n\). Thus we determine the order of growth, but not an asymptotic coefficient, which we believe to be closer to \(1\) if \(E_n/n\) converges.

The method of constructing SLR counterexamples in this note originates from the work of Goddyn and Wong \cite{GoddynWong2006} on ordinary LR. They changed selected velocities in the standard tight configuration \(1,\ldots,N\), while keeping the number of runners, the zero initial shifts, and the loneliness \(1/(N+1)\) unchanged. By such accelerations they obtained different tight configurations. Here, we perform replacements similar in spirit, but introducing shifts allows us to insert fewer runners than we delete. We start from velocities \(2,4,\ldots,2N\), all at shift zero, and replace nine runners with eight while ensuring that, at every time, some runner remains within \(1/(N+1)\) of the origin. Even one such replacement gives a counterexample to SLR, while performing linearly many compatible replacements gives the linear lower bound on \(E_n\). To deal with simultaneous compatible replacements, we introduce the paradigm of local uniform rules, and show that our eight-for-nine replacement is an example of such a rule.

In light of this work, the technology of accelerations  could have provided a route to disproving SLR, if one had been looking for a disproof rather than a proof. The geometric approach enabled Blanco, Criado and Santos to find counterexamples through an algorithmic search, even though they say they had in fact hoped to remove an assumption from an existing finite-checking theorem for SLR \cite[Section~1]{BlancoCriadoSantos2026}.

\paragraph{Acknowledgments and funding.}
I am grateful to Karim Adiprasito, Enis Kaya, Evrydiki Nestoridi, Stavros Papadakis, Vasiliki Petrotou and Christos Tatakis for organizing Combinatorics and Geometry in Mytilene, where I heard an inspiring talk on SLR, and to Paco Santos for giving that talk. I am grateful to Andrii Bondarenko for useful discussions. I am grateful to Matthias Beck and Matthias Schymura for their feedback on the first draft of this note. This work was funded by the Deutsche Forschungsgemeinschaft (DFG, German Research Foundation) -- SFB-Gesch\"aftszeichen 1624 -- Projektnummer 506632645.

\paragraph{AI disclosure.} The author believed that a linear lower bound on \(E_n\) should exist. When asked to find it, Astra produced a horrendous construction with constant \(1/15482879\) and an unreadable proof. A subsequent chain of human-machine interactions resulted in the simple argument explained in this note. The idea of uniform rules belongs to the author.

\section{The problem statement and the theorems}
We fix one runner at position $0$ with velocity $0$, and consider all possible configurations of $n$ more runners with distinct positive integer velocities $v_1,\ldots,v_n$ and initial shifts $s_1,\ldots,s_n\in\R/\Z$. For $x\in\R$, let $\norm{x}$ denote the distance to the closest integer. Define $E_n$ by
\[
\frac{1}{n+1+E_n}
=\inf_{v,s}\;\max_{t\in\R/\Z}\;\min_{1\le i\le n}\norm{s_i+v_it}.
\]

For velocities $1,\ldots,n$ with zero initial shifts, \( \max_{t\in\R/\Z}\;\min_{1\le i\le n}\norm{s_i+v_it} \) is exactly \(1/(n+1)\), so $E_n\ge0$. If the shifted lonely runner conjecture
(SLR) were true, then $E_n$ would be zero. It is therefore a natural measure of the failure of SLR, and it is interesting to know how \(E_n\) behaves when \(n\) grows. The upper bound on \(E_n \) is \(n-1\). Indeed, call a configuration a {\em \(\delta\)-cover} if for every $t$, some runner satisfies
$\norm{s_i+v_it}\le\delta$, i.e., at every time there is a runner not further than \(\delta\) from the stationary runner at \(0\). For every \(\delta > 1/(n+1+E_n)\), a \(\delta\)-cover exists, so \(E_n > n-1\) would imply a \(\delta\)-cover with \(\delta<1/(2n)\), which is not possible, since \(n\) runners can only cover a fraction \(2n\delta < 1\) of the time. 

The main result of this note is a positive \(n\)-linear lower bound on \(E_n\).

\begin{theorem}\label{thm:linear-growth}
For every $n \geq 1$,
\[
E_n\ge\lfloor n/287\rfloor.
\]
\end{theorem}

Our constructions actually imply a lower threshold after which SLR fails with a guarantee.

\begin{theorem}\label{thm:threshold}
For every \(n \geq 95\), \(E_n \geq 1\).
\end{theorem}

Since Blanco, Criado and Santos \cite[Proposition~5.13]{BlancoCriadoSantos2026} construct counterexamples for all \(5 \leq n \leq 17\), finding counterexamples for the remaining \(18 \leq n \leq 94\) would establish that SLR is false for all \(n \geq 5\), answering the first part of Zhang's Question~9.4 \cite{Zhang2026}.

\section{Uniform rules}
We now introduce our main tool in establishing the linear lower bound.

\begin{definition}
A \emph{uniform rule} consists of

\begin{itemize}
    \item \(R=\{r_i\}\), a finite set of even velocities to remove;
    \item \(S=\{(s_i,\alpha_i+\beta_i\delta)\},
\) a finite set of even velocities to rephase;
\item and
\(
A=\{(a_i,\gamma_i+\eta_i\delta)\},
\) a finite set of odd velocities to add.
\end{itemize}
It is required that, for every \(N\) and every odd \(b\) for which all
velocities \(r_i b\) and \(s_i b\) occur in the starting configuration \(2,4,\ldots,2N\),
removing the runners \(r_i b\), rephasing the runners \(s_i b\) by
\(\alpha_i+\beta_i\delta\), and adding the runners \(a_i b\) with shifts
\(\gamma_i+\eta_i\delta\) leaves a \(\delta\)-cover, for $\delta = 1/(N+1)$.

A uniform rule has {\em gain} $|R|-|A|$ and it is {\em runner-saving} if its gain is positive.
\end{definition}

Any application of a runner-saving uniform rule at an admissible scale \(b\)
produces a counterexample to SLR. In fact, a stronger statement is true and straightforward.

\begin{proposition}
\label{startingpoint}
Let \(g>0\) be the gain of a uniform rule, and let \(L\) be the largest
velocity removed or rephased by the rule. Then \(E_n\ge g\) for every
\(n\ge L/2-g\).
\end{proposition}

For our purposes, we need better control
over simultaneous applications of a uniform rule; this will be available for \emph{local} uniform
rules, defined below. In the starting unshifted configuration with velocities
\(2,4,\ldots,2N\), we say that the runner with velocity \(q\) is
\emph{responsible} for a time \(t\) if \(q\) is the smallest velocity satisfying
\(\|qt\|\le\delta\).

\begin{definition}
A uniform rule is called \emph{local} if, for every \(N\), every admissible
odd \(b\), and every time \(t\), whenever a runner removed or rephased by the
rule at scale \(b\) is responsible for \(t\), some runner rephased or added by
the same application of the rule is within distance \(\delta\) of the origin
at time \(t\).
\end{definition}

Runner-saving local uniform rules will give us positive $n$-linear lower bounds on $E_n$. For our main theorem, we need one more bit of notation. Let \(P\) be a positive integer, and list the positive integers coprime to
\(P\) as \(b_1<b_2<\cdots\). We write
\(\sigma^*(P):=\inf_{j\ge1}j/b_j\). This is the modified Schnirelmann
density of the sequence of positive integers coprime to \(P\), in the
terminology of Stalley~\cite{Stalley1955}. Equivalently, \(\sigma^*(P)\)
is the largest constant such that \(b_j\le j/\sigma^*(P)\) for every
\(j\ge1\).

\begin{theorem}
\label{main}
Suppose that there exists a local uniform rule of positive gain \(g\). Let
\(L\) be the largest even velocity removed or rephased by the rule, and let
\(P\) be the product of the distinct primes dividing the velocities occurring
in the rule. Then, for every \(n\ge1\),
\[
E_n\ge
g\left\lfloor
\frac{2\sigma^*(P)n}{L-2g\sigma^*(P)}
\right\rfloor.
\]
\end{theorem}

\begin{proof}
List the positive integers coprime to \(P\) as
\(b_1<b_2<\cdots\). By the definition of \(\sigma^*(P)\), we have
\(b_j\le j/\sigma^*(P)\) for every \(j\).

All $b_i$ are odd. Applications of the rule at distinct scales \(b_i\) and \(b_j\) involve
disjoint sets of velocities. Indeed, every prime dividing a velocity in the
rule divides \(P\), whereas \(b_i\) and \(b_j\) are coprime to \(P\), so
unique factorization determines both the scale and the unscaled velocity.

Locality implies that any finite collection of these applications may be
performed simultaneously. To see this, fix a time \(t\), and let \(q\) be
the runner responsible for \(t\) in the original configuration. If \(q\)
is unaffected, then she remains present and covers \(t\). Otherwise, \(q\)
is affected by a unique application of the rule, and locality provides a
runner from that same application who covers \(t\). This runner remains
present because the applications have disjoint velocity sets.

Given \(n\), put
\[
m=
\left\lfloor
\frac{2\sigma^*(P)n}{L-2g\sigma^*(P)}
\right\rfloor
\qquad\text{and}\qquad
N=n+gm.
\]
If \(m=0\), the result follows from \(E_n\ge0\), so assume \(m\ge1\). We have \(b_m\le m/\sigma^*(P)\), and 
\[
Lb_m\le\frac{Lm}{\sigma^*(P)}
\le 2n+2gm=2N.
\]
Thus the first \(m\) applications are all admissible in the starting
configuration with velocities \(2,4,\ldots,2N\).

Applying them simultaneously saves \(gm\) runners and leaves exactly
\(N-gm=n\) runners, while preserving coverage at radius \(1/(N+1)\).
Therefore \(E_n\ge N-n=gm\), which finishes the proof.
\end{proof}

\section{A runner-saving local uniform rule}
We now construct a runner-saving local uniform rule. In the notation of the previous section, set

    \[
    R=\{2^3,2^5,3\cdot2^6\};
    \]

    \[
    S=
    \left\{
    \left(3\cdot2^j,\frac{1+\delta}{2}\right):2\le j\le5
    \right\}
    \cup
    \left\{
    (2^4,-\delta),(2^6,-\delta)
    \right\};
    \]

    \[
    A=
    \left\{
    \left(3,\frac14\right),
    \left(3^2,\frac14\right)
    \right\}.
    \]

\begin{theorem}
    The assignment $(R,S,A)$ above is a local uniform rule of gain $1$.
\end{theorem}

\begin{proof}
Fix \(N\), an admissible odd \(b\), and a time \(t\). Let \(q\) be the runner responsible for \(t\), and write
\[
qt=k+\delta z,\qquad k\in\Z,\qquad |z|\le1.
\]
If the runner with speed \(q\) was not removed or rephased, she still covers \(t\). Otherwise \(q\) is one of the velocities removed or rephased by the rule at scale \(b\). Every such velocity is divisible by \(4\), so \(k\) is odd: otherwise dividing by \(2\) would give \((q/2)t=(k/2)+(\delta/2)z\) with \(q/2\) even and \(k/2\) an integer, so the runner with even velocity \(q/2\) would satisfy \(\norm{(q/2)t}\le\delta/2\), contradicting the choice of the smallest \(q\) in the assignment of responsibility.

We now consider three cases.

\emph{Case 1: \(q\in\{2^3b,2^4b,2^5b,2^6b\}\).}
Then \(3q/2\in\{3\cdot2^2b,3\cdot2^3b,3\cdot2^4b,3\cdot2^5b\}\), whose elements are velocities of runners rephased to \((1+\delta)/2\). At time \(t\) the runner with velocity \(3q/2\) is at position
\[
\frac{3qt}{2}+\frac{1+\delta}{2}
=\frac{3k}{2}+\frac{3\delta z}{2}+\frac{1+\delta}{2}
=\frac{3k+1}{2}+\frac{\delta(3z+1)}{2}.
\]
As \(k\) is odd, this is \(\delta(3z+1)/2\) modulo an integer, and \(\norm{\delta(3z+1)/2}\leq\delta\) when \(z\in[-1,1/3]\).

At the same time \(t\), the runner with velocity \(rq\), with \(r=2\) for \(q=2^3b,2^5b\) and \(r=1\) for \(q=2^4b,2^6b\), is also rephased by the rule, with velocity \(2^4b\) or \(2^6b\) and shift \(-\delta\). By a computation similar to the above, her position modulo an integer is \(\delta(rz-1)\), with \(\norm{\delta(rz-1)}\leq\delta\) for \(z\in[0,1]\). Together, \([-1,1/3]\) and \([0,1]\) cover all of \([-1,1]\).

\emph{Case 2: \(q\in\{3\cdot2^3b,3\cdot2^4b,3\cdot2^5b,3\cdot2^6b\}\).}
Then \(q/2\in\{3\cdot2^2b,3\cdot2^3b,3\cdot2^4b,3\cdot2^5b\}\), whose elements are velocities of runners rephased to \((1+\delta)/2\). By a computation similar to the above, at time \(t\) such a runner is at position \(\delta(1+z)/2\) modulo an integer, with \(\norm{\delta(1+z)/2}\leq\delta\) for any \(z\in[-1,1]\).

\emph{Case 3: \(q=3\cdot2^2b\).}
The runners with velocities \(q/4=3b\) and \(3q/4=3^2b\) are added with shift \(1/4\), so at time \(t\) they are at positions
\[
3bt+\frac14=\frac{k+1}{4}+\frac{\delta z}{4}
\qquad\text{and}\qquad
3^2bt+\frac14=\frac{3k+1}{4}+\frac{3\delta z}{4}.
\]
For odd \(k\), one of \((k+1)/4\) and \((3k+1)/4\) is an integer: the first when \(k\equiv3\pmod4\), the second when \(k\equiv1\pmod4\). The corresponding runner is within \(3\delta/4\) of an integer.

Thus, whenever a runner removed or rephased by the rule is responsible for \(t\), a runner rephased or added by the rule covers \(t\). Hence the rule is local and the resulting configuration remains a \(\delta\)-cover; as 3 runners are removed and 2 are added, we have a runner-saving local uniform rule.
\end{proof}

For this rule, the gain is \(g=1\), the largest removed or rephased velocity is \(L=192\), and the primes occurring in the rule are \(2\) and \(3\), so \(P=6\). The positive integers coprime to \(6\) satisfy \(b_j\le3j\), and \(j/b_j\) tends to \(1/3\); hence \(\sigma^*(6)=1/3\). Theorem~\ref{main} now gives
\[
E_n\ge
\left\lfloor
\frac{2n/3}{192-2/3}
\right\rfloor
=
\left\lfloor\frac{n}{287}\right\rfloor.
\]
Proposition~\ref{startingpoint} gives \(E_n\ge1\) whenever \(n\ge192/2-1=95\). This proves Theorems~1 and~2.

\section{Discussion}

The elementary upper bound $E_n\le n-1$, together with our construction,
shows that $E_n$ grows linearly in $n$. We make no optimality claim for
the coefficient $1/287$ in our lower bound. In fact, Table \ref{tab:additional-local-uniform-rules} lists some additional local uniform rules, with the biggest constant $1/224$. We decided not to sacrifice brevity for an insignificant constant improvement, since the small $n$ values appearing in \cite{BlancoCriadoSantos2026} give the impression of much quicker growth. We expect that lower bounds of the form $E_n\ge Cn-o(n)$ exist with $C$ much closer to $1$, although for significant improvement one probably needs techniques beyond our Theorem \ref{main}. We actually know nothing that would exclude even the possibility that the leading constant is \(1\). Indeed, excluding this possibility would solve the analogous problem for the ordinary LR, where it is notoriously difficult: see for example \cite[Theorem~1.2]{Tao2018} and \cite[Theorem~1.3]{Bedert2025} for sublinear improvements that were nevertheless harder to obtain than the result in this note.

\begin{table}[htbp]
\centering
\normalsize
\setlength{\tabcolsep}{4pt}
\renewcommand{\arraystretch}{1.18}
\caption{Some more local uniform rules, in shorthand notation \([X]_{\theta}:=\{(x,\theta):x\in X\} \).}

\label{tab:additional-local-uniform-rules}

\begin{tabular*}{\linewidth}{@{\extracolsep{\fill}}|c|l|c|c|c|c|c|}
\hline
Replacement
& \multicolumn{1}{c|}{Rule}
& \(g\)
& \(L\)
& \(P\)
& \(\sigma^*(P)\)
& \(C\)
\\[3pt]
\hline

\(19\to17\)
&
\begin{tabular}[t]{@{}r@{\;}l@{}}
\(R={}\) & \(\{8,32,36,40,72,180,192,240\},\)\\
\(S={}\) & \([\{12,24,48,60,96,108,120\}]_{\frac{1+\delta}{2}}\)\\
         & \({}\cup[\{16,64,80\}]_{-\delta}
             \cup[\{216\}]_{-2\delta},\)\\
\(A={}\) & \([\{3,9,27\}]_{\frac{1+\delta}{4}}\)\\
         & \({}\cup[\{15,45,135\}]_{\frac{3+\delta}{4}}.\)
\end{tabular}
& \(2\)
& \(240\)
& \(30\)
& \(\frac4{15}\)
& \(\frac1{224}\)
\\[3pt]
\hline

\(24\to21\)
&
\begin{tabular}[t]{@{}r@{\;}l@{}}
\(R={}\) & \(\{4,16,20,36,40,64,72,240,384,432\},\)\\
\(S={}\) & \([\{12,24,48,96,192\}]_{\frac{1+\delta}{2}}\)\\
         & \({}\cup[\{60,108,120,216\}]_{\frac{1-\delta}{2}}\)\\
         & \({}\cup[\{8,128\}]_{-\delta}
             \cup[\{80,144\}]_{\delta}\)\\
         & \({}\cup[\{32\}]_{-\frac{3\delta}{2}},\)\\
\(A={}\) & \([\{3,45,81\}]_{\frac{1-\delta}{4}}
             \cup[\{5,9\}]_{\frac{1+\delta}{4}}\)\\
         & \({}\cup[\{15,27\}]_{\frac14}.\)
\end{tabular}
& \(3\)
& \(432\)
& \(30\)
& \(\frac4{15}\)
& \(\frac1{269}\)
\\[3pt]
\hline

\(9\to8\)
&
\begin{tabular}[t]{@{}r@{\;}l@{}}
\(R={}\) & \(\{8,32,192\},\)\\
\(S={}\) & \([\{12,24\}]_{\frac{1+\delta}{2}}
             \cup[\{48,96\}]_{\frac{1-\delta}{2}}\)\\
         & \({}\cup[\{16\}]_{-\delta}
             \cup[\{64\}]_{\delta},\)\\
\(A={}\) & \([\{3,9\}]_{\frac14}.\)
\end{tabular}
& \(1\)
& \(192\)
& \(6\)
& \(\frac13\)
& \(\frac1{287}\)
\\[3pt]
\hline

\(16\to14\)
&
\begin{tabular}[t]{@{}r@{\;}l@{}}
\(R={}\) & \(\{16,40,64,240,384\},\)\\
\(S={}\) & \([\{24,48,60,96,120,192\}]_{\frac{1-\delta}{2}}\)\\
         & \({}\cup[\{36\}]_{\frac{1+\delta}{2}}
             \cup[\{32\}]_{\frac{3\delta}{2}}\)\\
         & \({}\cup[\{72\}]_{-2\delta}
             \cup[\{80,128\}]_{\delta},\)\\
\(A={}\) & \([\{15,27,45\}]_{\frac{1+\delta}{4}}.\)
\end{tabular}
& \(2\)
& \(384\)
& \(30\)
& \(\frac4{15}\)
& \(\frac1{359}\)
\\[3pt]
\hline

\(15\to14\)
&
\begin{tabular}[t]{@{}r@{\;}l@{}}
\(R={}\) & \(\{8,36,40,48,72,180,240\},\)\\
\(S={}\) & \([\{12,24,60,108,120\}]_{\frac{1+\delta}{2}}\)\\
         & \({}\cup[\{16,80\}]_{-\delta}
             \cup[\{216\}]_{-2\delta},\)\\
\(A={}\) & \([\{3,9,27\}]_{\frac{1+\delta}{4}}\)\\
         & \({}\cup[\{15,45,135\}]_{\frac{3+\delta}{4}}.\)
\end{tabular}
& \(1\)
& \(240\)
& \(30\)
& \(\frac4{15}\)
& \(\frac1{449}\)
\\[3pt]
\hline

\end{tabular*}
\end{table}

\bibliographystyle{plain}
\bibliography{lonely_and_sad}

\end{document}